\documentclass[12pt,reqno]{amsart}

\usepackage{cite}
\usepackage[colorlinks,citecolor=blue,hypertexnames=false,urlcolor=blue]{hyperref}
\usepackage{amssymb,amsmath,amsthm,amsfonts,latexsym,eucal}
\usepackage{bm}
\usepackage{mathtools}
\usepackage{enumerate}
\usepackage{array,graphics,xcolor}
\allowdisplaybreaks
\usepackage{mathrsfs}
\usepackage{ulem}
\usepackage{mathtools}
\usepackage{blkarray, bigstrut}
\usepackage{csquotes}
\usepackage{ulem}

\numberwithin{equation}{section}

\allowdisplaybreaks

\newtheorem{thm}{Theorem}[section]
\newtheorem{crl}[thm]{Corollary}
\newtheorem{lma}[thm]{Lemma}
\newtheorem{ppsn}[thm]{Proposition}

\theoremstyle{definition}

\newtheorem{dfn}[thm]{Definition}
\newtheorem{xmpl}[thm]{Example}

\theoremstyle{remark}
\newtheorem{rmrk}[thm]{Remark}

\newcommand{\R}{\mathbb{R}}
\newcommand{\C}{\mathbb{C}}
\newcommand{\N}{\mathbb{N}}	

\newcommand{\hilh}{\mathcal{H}}

\newcommand{\bh}{\mathcal{B}(\hilh)}

\newcommand{\Sp}{\mathcal{S}}

\newcommand{\Tr}{\operatorname{Tr}}

\newcommand{\fW}{\mathfrak{W}}
\newcommand{\fC}{\mathfrak{C}}
\newcommand{\fA}{\mathfrak{A}}
\newcommand{\fQ}{\mathfrak{Q}}

\newcommand{\cM}{\mathcal{M}}
\newcommand{\cW}{\mathcal{W}}
\newcommand{\cR}{\mathcal{R}}

\begin{document}
	

\title[Noncommutative $L_p$-differentiability and trace formulae]{Noncommutative $L_p$-differentiability and trace formulae}

\author[Chattopadhyay]{Arup Chattopadhyay}

\address{Department of Mathematics, Indian Institute of Technology Guwahati, Guwahati, 781039, India}
\email {arupchatt@iitg.ac.in, 2003arupchattopadhyay@gmail.com}

\author[Coine]{Cl\'ement Coine}
\address{Normandie Univ, UNICAEN, CNRS, LMNO, 14000 Caen, France}	
\email{clement.coine@unicaen.fr}

\author[Giri]{Saikat Giri}
\address{Department of Mathematics, Indian Institute of Technology Guwahati, Guwahati, 781039, India}
\email{saikat.giri@iitg.ac.in, saikatgiri90@gmail.com}

\author[Pradhan]{Chandan Pradhan}
\address{Department of Mathematics and Statistics, University of New Mexico, Albuquerque, NM 87106, USA}
\email{chandan.pradhan2108@gmail.com, cpradhan@unm.edu}

\subjclass{47A55, 46L52, 47A56}
\keywords{semifinite von Neumann algebra, noncommutative $L_p$-space, multiple operator integrals, spectral shift function}

\begin{abstract}
Let $\cM$ be a semifinite von Neumann algebra equipped with a normal faithful semifinite trace $\tau$, and let $L_p(\cM)$ denote the associated noncommutative $L_p$-space for $1<p<\infty$. Let $n\in\N$ and let $a, b$ be $\tau$-measurable self-adjoint operators such that $b\in L_p(\cM)\cap L_{np}(\cM)$. For a function $f\in C^n(\R)$ whose derivatives $f^{(k)}$ are bounded for $1\le k\le n$, we prove that the map $\phi:t\in\R\mapsto f(a+tb)-f(a)$ is $n$-times differentiable in the $\|\cdot\|_{L_p}$-norm. This strengthens the corresponding result of de Pagter and Sukochev for $p\neq 2$ and extends it to higher-order derivatives. In addition, if $f^{(n)}\in C_0(\R)$ or $b\in \cM$, then $\phi^{(n)}$ is continuous on $\R$. Consequently, we extend the Potapov--Skripka--Sukochev higher-order trace formula from bounded $L_n$-perturbations to not necessarily bounded perturbations in $L_n(\cM)\cap L_{n^{2}}(\cM)$. Moreover, we show that this trace formula holds for a broader class of admissible functions than the classes previously considered in the literature.
\end{abstract}

\maketitle

\section{Introduction}
Let $\hilh$ be a complex separable Hilbert space and let $\cM$ be a von Neumann algebra acting on $\hilh$. Throughout the paper, we assume that $\cM$ is equipped with a semifinite faithful normal trace $\tau:\cM^+ \to [0,\infty]$. In this situation, the pair $(\cM,\tau)$ is called a semifinite von Neumann algebra. We denote by $\widetilde{\cM}$ the $*$-algebra of all $\tau$-measurable operators affiliated with $\cM$. 

Let $1 \le p < \infty$. The noncommutative $L_p$-space associated with $(\cM,\tau)$ is defined by
\[
L_p(\cM,\tau) := \left\{ a \in \widetilde{\cM} : \|a\|_p := \tau(|a|^p)^{1/p} < \infty \right\}.
\]
For $p=\infty$, we set $L^\infty(\cM,\tau)=\cM$ and equip it with the operator norm $\|\cdot\|_\infty$. In the particular case where $\cM=\bh$ and $\tau=\Tr$ is the canonical trace, the space $L_p(\cM,\tau)$ coincides with the Schatten $p$-class, denoted by $\mathcal{S}^p(\hilh)$. In the sequel we will denote $L_p(\cM,\tau)$ simply by $L_p(\cM)$. For a comprehensive exposition of noncommutative $L_p$-spaces, we refer to \cite{PiXu03, Sukochev-book}. 

\medskip

\noindent\underline{Higher-order $L_p$-differentiability:}

\smallskip

Let $f:\R\to\C$ be a sufficiently regular function, and let $a,b\in\widetilde{\cM}$ be self-adjoint operators with $b\in L_p(\cM)$. We consider the operator-valued function $\phi:\R\to L_p(\cM)$ defined by
\begin{equation}\label{op_func}
\phi(t)=f(a+tb)-f(a).
\end{equation}
We say that $\phi$ is {\it noncommutative $L_p$-differentiable} if the limit
\[
\lim_{s\to 0}\frac{\phi(t+s)-\phi(t)}{s}
\]
exists in the $L_p$-norm.

The problem of $L_p$-differentiability of operator functions of the form \eqref{op_func} is fundamental in perturbation theory. It originated in the work of Daletskii and Krein \cite{DaletskiiKrein1956} for the case $\cM=\bh$ (see also \cite{widom}), where the goal was to understand the analytic behavior of operator functions under perturbations. A major advance was achieved in \cite{BS1}, where differentiability results were established using the theory of {\it double operator integrals} (DOI), which has since become a central tool in this area.

The question of higher-order differentiability of \eqref{op_func} for $\cM=\bh$ was investigated in \cite{Pe06}. The explicit formulas for higher-order derivatives were obtained via {\it multiple operator integrals} (MOI), which generalize DOI. Since then, substantial progress has been made in the Schatten-class setting; see, for example, \cite{KiPoShSu14, CoLemSkSu19, LeSk20, Co22}. The following result for $\mathcal{S}^p(\hilh)$-perturbations were obtained in \cite{CLMMcd,LeSk20}.

\begin{thm}\label{Le_Sk_diff}
Let $1<p<\infty$, and let $a,b$ be self-adjoint operators in $\hilh$ with $b\in\Sp^p(\hilh)$. Let $n\in\N$, and let $f$ be an $n$-times continuously differentiable function on $\R$ such that $f^{(i)}$ is bounded for all $1\le i\le n$. Define
\[
\phi:t\in\R\mapsto f(a+tb)-f(a)\in\Sp^p(\hilh).
\]
Then $\phi$ is $n$-times differentiable in the $\|\cdot\|_p$-norm. For every integer $1\le k\le n$, the derivative $\phi^{(k)}$ is bounded on\/ $\R$ and satisfies
\[
\frac{1}{k!}\,\phi^{(k)}(t)
=\left[\Gamma^{a+tb,a+tb,\ldots,a+tb}\bigl(f^{[k]}\bigr)\right](b,\ldots,b),
\quad t\in\R.
\]
Here, $\Gamma^{a + tb, a + tb, \ldots, a+ tb}(f^{[k]})$ is given by Definition \ref{moi_coine}. Moreover, if $f^{(n)}$ is uniformly continuous on $\R$, then $\phi^{(n)}$ is continuous on $\R$.
\end{thm}

It was subsequently shown in \cite{Co22} that the continuity assumption on $f^{(n)}$ can be removed without affecting the $n$-times differentiability of $\phi$.

In the general setting of semifinite von Neumann algebras, the differentiability of the function \eqref{op_func} under noncommutative $L_p$-perturbations was studied in \cite{PaSu04}. To state their result, we first recall some notation. 

Let $D^n(\R)$ denote the space of all $n$-times differentiable complex-valued functions on $\R$, and let $C^n(\R)$ be its subspace of $n$-times continuously differentiable functions. We write $C_c^n(\R)$ for the subspace of compactly supported functions in $C^n(\R)$, $C(\R)=C^0(\R)$, and $C_c(\R)=C_c^0(\R)$. Moreover, $C_b(\R)$ denotes the space of bounded continuous functions on $\R$, and $C_b^n(\R)$ consists of all $f\in C^n(\R)$ such that $f',\ldots,f^{(n)}\in C_b(\R)$. We also let $C_0(\R)$ denote the space of continuous functions vanishing at infinity, and $C_0^n(\R)$ the space of all $f\in C^n(\R)$ with $f^{(n)}\in C_0(\R)$. Finally, $D^n_b(\R)$ denotes the space of all $n$-times differentiable functions $f$ such that $f^{(n)}$ is bounded.

\begin{thm}{\normalfont\cite{PaSu04}}\label{thm:PaSu}
Let $1<p<\infty$, and let $a,b\in\widetilde{\cM}$ be self-adjoint operators with $b\in L_p(\cM)$. Suppose that $f\in C_b^1(\R)$ and that $f'$ has bounded variation. Define
\[
\phi:t\in\R\mapsto f(a+tb)-f(a)\in L_p(\cM).
\]
Then $\phi$ is differentiable in the $L_p$-norm, and its derivative is given by
\[
\phi'(t)=T_{f^{[1]}}^{a+tb,a+tb}(b), \quad t\in\R.
\]
For $p=2$, the bounded variation assumption on $f'$ may be omitted. Here $T_{f^{[1]}}^{a+tb,a+tb}$ is given by Definition~\ref{moi_pss}.
\end{thm}

Higher-order noncommutative $L_p$-differentiability of \eqref{op_func} has also been studied in \cite{AzCaDoSu09, Nik23, Nik23jam} under the additional assumption that the perturbation is bounded, namely, $b\in L_p(\cM)\cap\cM$, and for substantially more restrictive classes of functions $f$. In particular, \cite{AzCaDoSu09} considers functions $f\in C^n(\R)$, such that the $j$-th derivative $f^{(j)}$, $j=0,1,\ldots,n$, is the Fourier transform of a finite measure on $\R$, while \cite{Nik23, Nik23jam} deal with $C^n$-functions whose $n$-th order divided differences belong to certain integral tensor spaces.

In the present paper, we revisit the problem of noncommutative $L_p$-differentiability of the operator function \eqref{op_func}. Our main goal is to establish a complete noncommutative $L_p$-analogue of Theorem~\ref{Le_Sk_diff}. In particular, we obtain higher-order differentiability results for unbounded $L_p$-perturbations in the setting of semifinite von Neumann algebras. Our main result on $L_p$-differentiability is as follows.

\begin{thm}\label{DiffLp0}
Let $1<p<\infty$, $n\in\N$, and let $f\in C_b^n(\R)$. Let $a,b\in\widetilde{\cM}$ be self-adjoint operators such that $b\in L_p(\cM)\cap L_{np}(\cM)$. Define
\[
\phi:t\in\R\mapsto f(a+tb)-f(a)\in L_p(\cM).
\]
Then $\phi$ is $n$-times differentiable on $\R$, and for every integer $1\le k\le n$,
\[
\frac{1}{k!}\,\phi^{(k)}(t)
= T_{f^{[k]}}^{a+tb,\ldots,a+tb}(b,\ldots,b),
\]
where $T_{f^{[k]}}^{a+tb,\ldots,a+tb}$ is given by Definition \ref{moi_pss}. Moreover, $\phi^{(n)}$ is continuous on $\R$ if any of the following holds: $f^{(n)} \in C_0(\R)$; $f^{(n)}$ is uniformly continuous on $\R$ and $b\in L_{p}(\cM)\cap L_{(n+1)p}(\cM)$; or $b\in L_{p}(\cM)\cap\cM$.
\end{thm}

Thus, our result extends Theorem~\ref{thm:PaSu} to higher orders and improves it for $p\neq2$, thereby providing a complete analogue of Theorem~\ref{Le_Sk_diff} in the setting of noncommutative $L_p$-perturbations. The above notion of differentiability is usually referred to as the G\^ateaux derivative of an operator function.  The stronger notion of Fr\'echet differentiability has also been studied for Schatten classes; see, for example, \cite{KiPoShSu14,LeSk20}. However,  the necessary and sufficient condition for Fr\'echet differentiability in $\mathcal{S}^p(\hilh)$ spaces established in \cite{CLMMcd} indicates that the condition $f \in C_b^n(\mathbb{R})$ with $f^{(n)}$ uniformly continuous is not sufficient to guarantee higher-order Fr\'echet differentiability on a general noncommutative $L_p$-space, see \cite[Comment $2$]{CLMMcd}.

\bigskip

\noindent\underline{Higher-order spectral shift functions:} 

\smallskip

The second part of this paper is devoted to the study of traces of noncommutative Taylor remainders and their integral representations in terms of spectral shift functions (SSFs). Let $f:\R \to \C$ be a sufficiently regular function, and let $a,b\in \widetilde{\cM}$ be self-adjoint operators. We define
\begin{align*}
\cR_n(f,a,b) := f(a+b) - \sum_{k=0}^{n-1} \frac{1}{k!} 
\left. \frac{d^k}{dt^k} \bigl[f(a+tb)\bigr] \right|_{t=0}.
\end{align*}
The operator $\cR_n(f,a,b)$ is referred to as the {\it noncommutative $n$-th Taylor remainder} associated with the operator-valued function $t \mapsto f(a+tb)$. Its well-definedness depends crucially on the existence of higher-order derivatives $\left. \frac{d^k}{dt^k} [ f(a+tb) ] \right|_{t=0}$, and hence on Theorem \ref{DiffLp0}. Understanding these properties is therefore essential for the analysis of $\cR_n(f,a,b)$.

Our main objective is to study trace formulas for $\cR_n(f,a,b)$; first, to ensure that $\cR_n(f,a,b)\in L_{1}(\mathcal{M})$, and then to express $\tau(\cR_n(f,a,b))$ in terms of higher-order spectral shift functions. These formulas generalize the classical first-order trace formula and play a fundamental role in perturbation theory of operators affiliated with semifinite von Neumann algebras.

\smallskip

The notion of first-order spectral shift function and the associated trace formula for $\cR_1(f,a,b)$, in the case $\cM=\bh$, originated in the work of Lifshitz \cite{Li52} in theoretical physics and was subsequently developed rigorously by Krein \cite{Kr53}. The resulting {\it Lifshitz--Krein spectral shift function} and the corresponding trace formula have since become central tools in operator theory, with far-reaching applications in noncommutative geometry \cite{ChCo97,Connes-Book}, scattering theory \cite{BiYa93,Yafaev}, index theory \cite{Push08, GeLaMaSuTo11}, and related areas.

The existence of higher-order spectral shift functions was established by Potapov, Skripka, and Sukochev. More precisely, the following result was proved in \cite{PoSkSu13}.

\medskip

\noindent{\bf The function class $\cW_n(\R)$.}
Let $n,k\in\N\cup \{0\}$ with $k\leq n$. 
Denote by $W_k(\R) \subseteq C^k(\R)$ the class of functions $f\in C^k(\R)$ such that the Fourier transform of $f^{(k)}$ is integrable. Define
\begin{align*}
\cW_n(\R)=\bigcap_{k=0}^n W_k(\R).
\end{align*}

\begin{thm}{\normalfont\cite[Theorem 1.1]{PoSkSu13}}\label{inv_thm}
Let $n \in \mathbb{N}$, and let $a,b$ be self-adjoint operators affiliated with $\mathcal{M}$ such that $b\in L_n(\cM)\cap \cM$. Then $\cR_n(f,a,b) \in L_1(\cM)\cap \cM$ and there exists a unique function 
$\eta_n=\eta_{n,a,b}\in L^1(\mathbb{R})$ satisfying
\begin{align}\label{PSS-formula-trace}
\tau(\cR_n(f,a,b)) = \int_{\mathbb{R}} f^{(n)}(t)\,\eta_n(t)\,dt,
\end{align}
for every $f\in\cW_n(\R)$ with the estimate
\begin{align*}
\|\eta_n\|_1 \leq c_n \|b\|_n^n.
\end{align*}
\end{thm}

The function $\eta_n$ is called the {\it $n$-th order spectral shift function}, and \eqref{PSS-formula-trace} is referred to as the {\it $n$-th order trace formula} associated with the Taylor remainder $\cR_n(f,a,b)$. In the case $\cM=\bh$, the existence of $\eta_1$ is due to Lifshitz and Krein, while the existence of $\eta_2$ was established by Koplienko \cite{Ko84}, who further conjectured the existence of $\eta_n$ for all $n\geq2$. After a partial solution in \cite{DySk-JFA, Sk-Indiana}, this longstanding conjecture was fully resolved in \cite{PoSkSu13}. We note that a recent preprint \cite{Oliver} establishes a connection between index theory and higher-order spectral shift functions for a concrete class of operators. Related trace formulas for $\cR_1(f,a,b)$ were also investigated in \cite{DySk14, PoSu14jst}. Indeed, in \cite{PoSu14jst} the Krein trace formula for $\cR_1(f,a,b)$ was established for perturbations $b\in L_1(\cM)$ (not necessarily bounded) and for functions $f$ whose derivative $f'$ has an integrable Fourier transform. To the best of authors' knowledge, this remains the strongest available result under $L_1$-perturbations.

Given the existence of $\eta_n$, it is natural to ask whether the assumptions on the operators $a$ and $b$, or the regularity conditions imposed on the scalar function $f$, can be weakened. This problem has attracted considerable attention. In the setting $\cM=\bh$, various extensions of \eqref{PSS-formula-trace} were obtained in \cite{ChSk18, PoSkSu15, Sk17Adv, NuSk22, NuSk23}, where $b$ is allowed to be bounded but not necessarily compact, subject to suitable summability conditions. More recently, symmetric and relative $\Sp^n(\hilh)$-perturbations (not necessarily bounded) were considered in \cite{AlPe25} for $n=1$ and in \cite{ChNuPr24} for $n\geq1$. For general semifinite von Neumann algebras, the cases where $b = b^* \in \cM$ and $(a-i)^{-1}$ is either $\tau$-compact or belongs to $L_n(\cM) \cap \cM$ in \eqref{PSS-formula-trace} were studied in \cite{Sk18, ChPrSk23}. However, higher-order trace formulas for not necessarily bounded perturbations $b\in L_n(\cM)$ have not been addressed to date.

The optimal class of scalar functions for which the Lifshitz--Krein trace formula (with $\cM=\bh$) holds was identified in \cite{Pe16}. For higher-order trace formulas ($n\geq2$), however, the determination of an optimal function class remains open. Notable progress was made for $n=2$ in \cite{CoLemSkSu19, ChCoGiPr24}, where the class of admissible functions was considerably enlarged to include the Besov space $B^2_{\infty 1}(\R)$, extending earlier results of \cite{Pe05}. For $n>2$, a trace formula for $\cR_n(f,a,b)$ (in an appropriate sense) was established in \cite{AlPe11} for $b\in \Sp^n(\hilh)$ and $f$ belonging to the Besov space $B^n_{\infty 1}(\R)$, which contains $\cW_n(\R)$.

In the present paper, working in the general setting of an arbitrary semifinite von Neumann algebra $\cM$ and perturbations $b\in L_n(\cM)\cap L_{n^{2}}(\cM)$, we extend the function class $\cW_n(\R)$ considered in \cite{PoSkSu13}. More precisely, in Theorem~\ref{main_thm_Lp} we introduce a new function class $\fW_n(\R)$ (see~\eqref{Funct_Class-1}) containing the class $\cW_n(\R)$ and extend Theorem~\ref{inv_thm} to perturbations $b\in L_n(\cM)\cap L_{n^{2}}(\cM)$ (not necessarily bounded) and functions $f\in\fW_n(\R)$, thereby closing the gap between bounded and unbounded perturbations in this setting. When specialized to the case $\cM=\bh$, this class is further enlarged to $\fQ_n(\R)$ (see \eqref{Funct_Class-2}). Within this refined framework, we establish a trace formula for a suitably modified version of $\cR_n(f,a,b)$, in the spirit of \cite{AlPe11}, valid for functions in the class $\mathcal{C}_n(\R)$ (see \eqref{new-functionclass}), which contains $B^n_{\infty1}(\R)$; see Theorem \ref{trace_formula_thm2} and Remark \ref{similar to peller}.

In the study of both differentiability and trace formulas, we make extensive use of the theory of multiple operator integrals. Multiple operator integrals were introduced in \cite{Pe06, AzCaDoSu09} as an extension of the theory of double operator integrals \cite{BS1, PaWiSu02}. Through numerous developments and applications over the past twenty years, these techniques have established themselves as potent and versatile tools in perturbation theory \cite{PoSkSu13, PoSkSu14, PoSkSu15, Sk17Adv, CoLemSkSu19, PoSuToZa19, AlPe25, ChCoGiPr24, ChNuPr24}. A detailed account of the theory of multiple operator integrals, together with its extensive applications in perturbation theory, can be found in the survey article \cite{Peller-survey} and the book \cite{SkTo_Book}. In this paper, along with a number of auxiliary technical results, we establish a strengthened perturbation formula for multiple operator integrals (Theorem~\ref{Prtb-Thm}), which is one of the key ingredients in the proof of our main results. The major difference from previously known versions is discussed prior to the statement of the theorem.

The paper is organized as follows. In Section~\ref{sec:prel}, we recall several auxiliary function classes and introduce the definition of multiple operator integrals together with their basic properties. In Section~\ref{sec:differentiation}, we prove Theorem~\ref{DiffLp0}. Section~\ref{sec:traceLp} is devoted to the study of trace formulas for noncommutative $L_p$-perturbations, while in Section~\ref{sec:traceSp} we investigate trace formulas for perturbations in the Schatten $p$-classes.

\section{Preliminaries}\label{sec:prel}
We recall some classes of functions that will be used throughout the paper.

\medskip

\noindent {\bf The class $\fA_n$.} Let $n \in \N$. We denote by $\fA_n$ the class of functions
$\varphi : \R^{n+1} \to \C$ admitting a representation of the form
\begin{align}\label{function_representation}
\varphi(\lambda_1, \ldots, \lambda_{n+1})= \int_{\Omega} \alpha_0(\lambda_1, \omega) \cdots \alpha_{n}(\lambda_{n+1}, \omega)
\, d\nu(\omega),
\end{align}
where $(\Omega,\nu)$ is a measure space with a complex measure $\nu$ and
\[
\alpha_i : \R \times \Omega \to \C, \qquad i=0,\ldots,n,
\]
are bounded measurable functions (with respect to the Borel $\sigma$-algebra on
$\R$) such that
\[
\int_{\Omega} \|\alpha_0(\cdot,\omega)\|_\infty \cdots \|\alpha_{n}(\cdot,\omega)\|_\infty
\, d|\nu|(\omega) < \infty.
\]

\medskip

\noindent {\bf The class $\mathfrak{C}_n$}.
Let $\fC_n$ denote the subset of $\fA_n$ of functions admitting the representation \eqref{function_representation},
where $\bigcup_{k=1}^{\infty} \Omega_k=\Omega$ for a growing sequence $\{\Omega_k\}_{k=1}^\infty$ of measurable subsets of $\Omega$
such that the families $\{\alpha_j(\cdot, \omega)\}_{\omega\in\Omega_k}$, $j=0,\ldots,n$ are uniformly bounded and uniformly continuous. A norm on $\fC_n$ is defined by
\[
\|\varphi\|_{\fC_n}:=\inf\int_{\Omega}
\|\alpha_0(\cdot,\omega)\|_\infty \cdots \|\alpha_{n}(\cdot,\omega)\|_\infty
\, d|\nu|(\omega) < \infty,
\]
where the infimum is taken over all possible representations \eqref{function_representation} with $\alpha_j, j=0,\dots,n$, as above. See \cite{PoSkSu13} for more informations about this space.

\medskip

\noindent{\bf The Besov space.}  Let $w_0 \in C^\infty(\R)$ be such that its Fourier transform is supported in $[-2,-1/2]\cup[1/2,2]$ and $\widehat w_0(y) + \widehat w_0(y/2) = 1$ for $1 \le y \le 2,$
and define $ w_k(x) = 2^k w_0(2^k x)$ for $ x \in \R,\, k \in \mathbb{Z},$ where $\mathbb{Z}$ is the set of integers. Then, following \cite{CoLemSkSu19,LeSk20,Pe06}, we define for $n\in\mathbb{N}\cup\{0\}$,
\[B^n_{\infty 1}(\mathbb{R})=
\left\{
f \in C^n(\mathbb{R}) :
\|f^{(n)}\|_\infty
+
\sum_{k\in\mathbb{Z}} 2^{nk} \|f * w_k\|_\infty
< \infty
\right\}.\]
For more background on Besov spaces, we refer to \cite{Pe76,St70,Tr83}.

\medskip

Recall that for $f\in D^{n}(\R)$, the $n$-th order divided difference $f^{[n]}:\R^{n+1}\to\C$ is defined recursively by
\[
f^{[0]}(\lambda_{0}) := f(\lambda_{0}), \qquad
f^{[n]}(\lambda_{0},\ldots,\lambda_{n}) := \lim_{\lambda\to\lambda_{n}}
\frac{f^{[n-1]}(\lambda_{0},\ldots,\lambda_{n-2},\lambda) - f^{[n-1]}(\lambda_{0},\ldots,\lambda_{n-1})}{\lambda-\lambda_{n-1}}.
\]
Moreover, if $f^{(n)}$ is bounded, then $f^{[n]}$ is also bounded.

\medskip

The following lemma collects results from \cite[p.~9]{SkTo_Book} and \cite[Theorem~4.3.4]{SkTo_Book}.

\begin{lma}\label{setinclussion}
Let $n\in\N$. Then the following inclusions hold:
\[
C_c^{n+1}(\R)\subseteq \cW_n(\R)\subseteq B_{\infty 1}^n(\R).
\]
Moreover, if $f\in B_{\infty 1}^n(\R)$, then $f^{[n]}\in \fC_n$.
\end{lma}

\medskip

We now recall a fundamental tool, namely multiple operator integrals, which are essential for the analysis of differentiability of operator functions and of trace formulas for Taylor remainders.

\smallskip

\noindent{\bf Multiple operator integral: Potapov-Skripka-Sukochev's approach.} Following \cite{PoSkSu13} and \cite{PoSuToZa19}, let us recall the definition and the result regarding the boundedness of multiple operator integrals associated to a divided difference. For a self-adjoint operator $a$ affiliated with $\mathcal{M}$, we will denote by $E_a$ its spectral measure, and we set $E_a^{l,m} = E_a\left(\left[ \frac{l}{m},\frac{l+1}{m} \right)\right)$, for every $m\in \mathbb{N}$ and $l\in \mathbb{Z}$.

\begin{dfn}\label{moi_pss}
Let $n\in\N$, and let $1\le p_1,\ldots,p_n\leq\infty$ be such that $\frac{1}{p}:=\sum_{k=1}^n \frac{1}{p_k} \leq 1$. Let $a_1,\ldots, a_{n+1}$ be self-adjoint operators affiliated with $\mathcal{M}$, $x_{k}\in L_{p_{k}}(\mathcal{M})$ for $1\le k\le n$, and let $\varphi:\R^{n+1}\to\C$ be a bounded Borel function. Assume that for every $\tilde{x} := (x_1,\ldots,x_n)$ and every $m\in \mathbb{N}$, the limit
\begin{align*}
S_{\varphi,m}(\tilde{x}):= \lim\limits_{N\to +\infty} \sum_{-N\leq l_1,\ldots, l_{n+1} \leq N} \varphi\left(\frac{l_1}{m},\ldots,\frac{l_{n+1}}{m} \right) E_{a_1}^{l_1,m}x_1 E_{a_2}^{l_2,m} \cdots E_{a_n}^{l_n,m} x_n E_{a_{n+1}}^{l_{n+1},m}
\end{align*}
exists in $L_p(\mathcal{M})$, that
$$S_{\varphi,m} : L_{p_1}(\mathcal{M}) \times \cdots \times L_{p_n}(\mathcal{M}) \to L_p(\mathcal{M})$$
is a bounded multilinear operator, and that the sequence $\{S_{\varphi,m}\}_{m\geq 1}$ converges pointwise to a multilinear operator denoted by $T_{\varphi}^{a_1, \ldots, a_{n+1}}$. In that case, according to Banach-Steinhaus theorem, $T_{\varphi}^{a_1,\ldots,a_{n+1}}$ is a bounded multilinear operator $L_{p_1}(\mathcal{M}) \times \cdots \times L_{p_n}(\mathcal{M}) \to L_p(\mathcal{M})$. 
\end{dfn}

Below, we state that for a large class of functions $\varphi$, the multilinear operator $T_\varphi$ is bounded.

\begin{thm}{\normalfont\cite[Theorem 5.3]{PoSkSu13} and \cite[Theorem 25]{PoSuToZa19}}\label{bddMOI}
Let $n\in\N$, and let $1<p_1,\ldots,p_n<\infty$ be such that $0<\frac{1}{p}:=\sum_{k=1}^n \frac{1}{p_k}<1$. Let $a_1,\ldots, a_{n+1}$ be self-adjoint operators affiliated with $\mathcal{M}$, and let $f\in C^n(\R)$ be such that $f^{(n)}$ is bounded. Then $T_{f^{[n]}}^{a_1,\ldots,a_{n+1}}:L_{p_1}(\mathcal{M})\times \cdots\times L_{p_n}(\mathcal{M})\to L_{p}(\mathcal{M})$ is bounded and there exists a constant $c_{p,n}>0$ such that
$$\left\| T_{f^{[n]}}^{a_1,\ldots,a_{n+1}} \right\| \leq c_{p,n} \|f^{(n)}\|_{\infty}.$$
\end{thm}

\begin{rmrk}
In \cite[Theorem 5.3]{PoSkSu13}, the theorem is proved for the case of the same spectral measure. In \cite[Theorem 2.2]{LeSk20}, it is extended to different spectral measures, but in the setting of Schatten classes, while in \cite[Theorem 25]{PoSuToZa19}, it is stated for different spectral measures in the general case.
\end{rmrk}

The following result provides a simple explicit formula for the multilinear operator $T_\varphi$ when $\varphi\in\fC_n$.

\begin{lma}{\normalfont\cite[Lemma 3.5]{PoSkSu13}}\label{nice_moi_bdd}
Let $1\leq p_j\leq \infty$, with $1\leq j\leq n$, be such that
\[ 0\leq \frac{1}{p} := \frac{1}{p_1}+\cdots+\frac{1}{p_n}\leq 1.\]
Let $a_1,\ldots,a_{n+1}$ be self-adjoint operators affiliated with $\cM$. For every $\varphi\in \mathfrak{C}_n$, the operator
$$T_\varphi^{a_1,\ldots,a_{n+1}} : L_{p_1}(\cM)\times\cdots\times L_{p_n}(\cM) \to L_p(\cM)$$
exists and is bounded  with
\[\|T_\varphi^{a_1,\ldots,a_{n+1}}\| \leq \|\varphi\|_{\mathfrak{C}_n}.\]
Moreover, given the decomposition \eqref{function_representation} of the function $\varphi$, the operator $T_\varphi^{a_1,\ldots,a_{n+1}}$ can be represented as the Bochner integral
\begin{align}\label{moi-peller}
T_\varphi^{a_1,\ldots,a_{n+1}}(b_1,b_2,\ldots,b_n) = \int_{\Omega} \alpha_0(a_1,\omega)\,b_1 \alpha_1(a_2,\omega)\,b_2\cdots b_n\alpha_n(a_{n+1},\omega) \,d\nu(\omega).
\end{align}
\end{lma}
The right-hand side of \eqref{moi-peller} is known as a multiple operator integral. It was introduced by Peller \cite{Pe06} for $\cM=\bh$, and independently by Azamov et al. \cite{AzCaDoSu09} for general $\cM$.

The following estimate can be obtained from \cite[Theorem 5.3]{PoSkSu13} and Lemma~\ref{nice_moi_bdd}, together with the cyclicity of the trace and H\"older's inequalities for noncommutative $L_p$-space elements, by following the same line of argument as in the proof of \cite[Eq.~(5.33)]{PoSkSu13} (see also \cite[Theorem~2.5]{ChPrSk23}).

\begin{crl}\label{cr:trace-bdd}
Let $n\in\N$. For $n\geq 2$, let $1<p_1,\ldots,p_n<\infty$ satisfy $\sum_{k=1}^n \frac{1}{p_k}=1,$ and for $n=1$ set $p_1=1$. Let $a_1,\ldots,a_n$ be self-adjoint operators affiliated with $\cM$, and let
$b_i\in L_{p_i}(\cM)$ for $1\leq i\leq n$. Then there exists a constant $c_n>0$ (depending upon $p_1,\ldots, p_n$) such that, for every $f \in C^n(\R)$ with $f^{[n]} \in \mathfrak{C}_n$, 
\begin{align*}
\left|\tau\!\left(T_{f^{[n]}}^{a_1,\ldots,a_n,a_1}\big(b_1,\ldots,b_n\big)\right)\right|\leq c_n\,\|f^{(n)}\|_\infty\,\|b_1\|_{p_1}\cdots\|b_n\|_{p_n}.
\end{align*}
\end{crl}

\section{Differentiability of operator functions in noncommutative $L_p$-spaces}\label{sec:differentiation}

We first establish perturbation formulas for multiple operator integrals, which are essential for proving the differentiability results in this section. In view of \cite[Theorem~7.4]{PaWiSu02} and \cite[Theorem~7]{PoSu12}, we have the following:

\begin{thm}\label{Perturbation-Theorem}
Let $1<p<\infty$. Suppose that $a,b$ are self-adjoint operators affiliated with $\cM$ and $a-b\in L_{p}(\mathcal{M})$. Let $f\in\textnormal{Lip}(\R)$, then
\begin{align*}
T^{a,b}_{f^{[1]}}(a-b)=f(a)-f(b) \ \in L_p(\mathcal{M}).
\end{align*}
\end{thm}

\begin{lma}\label{MOIweakstarcont}
Let $n\in\N$ and let $1<p_0,\ldots, p_{n-1}<\infty$ be such that $0<\frac{1}{p} := \sum_{k=0}^{n-1} \frac{1}{p_k} <1$. Suppose that $a_1,\ldots,a_{n+1}$ are self-adjoint operators affiliated with $\mathcal{M}$ and  $x_i \in L_{p_i}(\mathcal{M}), 1\leq i \leq n-1$.  Let $f\in C^n(\mathbb{R})$ with bounded $n$-th derivative $f^{(n)}$. Then, the map
$$\begin{array}{lrcl}
T : & L_{p_0}(\mathcal{M}) & \longrightarrow & L_{p}(\mathcal{M}) \\
& x & \longmapsto & T_{f^{[n]}}^{a_1,\ldots, a_{n+1}}(x_1, \ldots,x_{i-1},x,x_i,\ldots,x_{n-1})
\end{array}$$
is $w^*$-continuous, that is, if $(x_i)_i \to x$ weakly in $L_{p_0}(\mathcal{M})$, then $(T(x_i))_i \to T(x)$ weakly in $L_{p}(\mathcal{M})$.
\end{lma}
		
\begin{proof}
It is sufficient to show that $T$ is the adjoint of the bounded map
$$\begin{array}{lrcl}
S : & L_{q}(\mathcal{M}) & \longrightarrow & L_{r}(\mathcal{M}) \\
& y & \longmapsto & T_{f^{[n]}}^{a_{i+1},\ldots, a_{n+1}, a_1, \ldots, a_i}(x_i, \ldots, x_{n-1},y,x_1,\ldots,x_{i-1}),
\end{array}$$
where $\frac{1}{p}+\frac{1}{q}=1$ and $\frac{1}{p_0}+\frac{1}{r} = 1$. Let $x\in L_{p_0}(\mathcal{M})$ and $y\in L_{q}(\mathcal{M})$. Set, for $m,N \geq 1$,
\begin{align*}
&T_{m,N}(x)\\
&:=\sum_{-N\leq l_1,\ldots, l_{n+1} \leq N} f^{[n]}\left(\frac{l_1}{m},\ldots,\frac{l_{n+1}}{m} \right) E_{a_1}^{l_1,m}x_1 E_{a_2}^{l_2,m} \cdots E_{a_i}^{l_i,m} x E_{a_{i+1}}^{l_{i+1},m}  \cdots E_{a_n}^{l_n,m} x_{n-1} E_{a_{n+1}}^{l_{n+1},m}
\end{align*}
and
\begin{align*}
&S_{m,N}(y)\\
&:=\sum_{-N\leq l_1,\ldots, l_{n+1} \leq N} f^{[n]}\left(\frac{l_1}{m},\ldots,\frac{l_{n+1}}{m} \right) E_{a_{i+1}}^{l_{i+1},m}  \cdots E_{a_n}^{l_n,m} x_{n-1} E_{a_{n+1}}^{l_{n+1},m} y E_{a_1}^{l_1,m} x_1 E_{a_2}^{l_2,m} \cdots E_{a_i}^{l_i,m}.
\end{align*}
Using the linearity and the cyclicity of the trace, we get
$$
\tau(T_{m,N}(x)y) = \tau(x S_{m,N}(y)).
$$
Now, taking the limit as $N\to +\infty$, and then as $m\to +\infty$, yields
$$\tau(T(x) y) = \tau(x S(y)),$$
whence the result.
\end{proof}
		
The above implies the following perturbation formula. Note that a similar result has been obtained recently in \cite[Proposition 4.11]{ChenHong26},  providing an extension of the earlier result \cite[Theorem~28]{PoSuToZa19} in the special case where the operators $a$ and $b$ appearing in the formula are assumed to be in an appropriate $L_p$-space, while in our result, we solely require $b-a \in L_p(\cM)$.

\begin{thm}\label{Prtb-Thm}
Let $k\in\N$ and let $1 < p_0,\ldots, p_k < \infty$ be such that $  0 < \sum_{j=0}^k \frac{1}{p_j} <1$. Let $a,b,a_1,\ldots, a_k$ be self-adjoint operators affiliated with $\mathcal{M}$ and such that  $b-a\in L_{p_0}(\mathcal{M})$. Let $x_i \in L_{p_i}(\mathcal{M}), 1\leq i \leq k$. Let $f\in C^{k+1}(\mathbb{R})$ be such that $f^{(k)}, f^{(k+1)} \in C_b(\R)$. Then, for every $1 \le j \le k+1$ the following formula holds:
\begin{align*}
&T_{f^{[k]}}^{a_1, \ldots, a_{j-1}, b, a_j, \ldots, a_k}(x_1, \ldots, x_k) - T_{f^{[k]}}^{a_1, \ldots, a_{j-1}, a, a_j, \ldots, a_k}(x_1, \ldots, x_k) \\
&\ =T_{f^{[k+1]}}^{a_1, \ldots, a_{j-1}, b, a, a_j, \ldots, a_k}(x_1, \ldots, x_{j-1}, b-a, x_j, \ldots, x_k).
\end{align*}
\end{thm}
		
\begin{proof}
To simplify the notations, we only prove the formula for $j=1$. The general case is similar. Define, for every $N\in\mathbb{N}$ and every $(\lambda_0,\ldots,\lambda_{k+1}) \in\mathbb{R}^{k+2}$,
$$
\phi_N(\lambda_0,\ldots,\lambda_{k+1})=\chi_{[-N,N]}(\lambda_0)\chi_{[-N,N]}(\lambda_1),\quad\phi_0(\lambda_0,\ldots,\lambda_{k+1})=\lambda_0-\lambda_1,
$$
$$
\psi_1(\lambda_0,\ldots,\lambda_{k+1})=f^{[k]}(\lambda_0,\lambda_2,\ldots,\lambda_{k+1})\quad\text{and}\quad\psi_2(\lambda_0,\ldots,\lambda_{k+1})=f^{[k]}(\lambda_1,\lambda_2,\ldots,\lambda_{k+1}).
$$
It follows from the definition of $f^{[k+1]}$ that
\begin{align}\label{step1PF}
f^{[k+1]}\phi_N \phi_0 = \psi_1 \phi_N - \psi_2 \phi_N.
\end{align}	
Next, from the definition of double operator integrals we have, for every $x\in L_2(\mathcal{M})$,
\begin{align}\label{step2PF}
T_{\phi_N}^{b,a}(x) = b_Nxa_N \quad \text{and} \quad T_{\phi_N \phi_0}^{b,a}(x) = bb_Nxa_N - b_Nxa_Na
\end{align}
where $a_N:=\chi_{[-N,N]}(a)$ and $b_N:=\chi_{[-N,N]}(b)$, so that $bb_N$ and $a_Na$ are elements of $\mathcal{M}$. Hence, according to \cite[Lemma 3.2 (iv)]{PoSkSu13} and \eqref{step2PF}, for every $\tau$-finite projection $p$,
\begin{align}\label{step3PF}
T_{f^{[k+1]}\phi_{N}\phi_{0}}^{b, a, a_1, \ldots, a_k}(p, x_1, \ldots, x_k) = T_{f^{[k+1]}}^{b, a, a_1, \ldots, a_k}(bb_Npa_N-b_Npa_Na, x_1, \ldots, x_k),
\end{align}
and
\begin{align}\label{step4PF}
\begin{split}
T_{\psi_1 \phi_N}^{b, a, a_1, \ldots, a_k}(p, x_1, \ldots, x_k)
& = T_{\psi_1}^{b, a, a_1, \ldots, a_k}(b_Npa_N, x_1, \ldots, x_k)\\
& = T_{f^{[k]}}^{b, a_1, \ldots, a_k}(b_Npa_Nx_1, x_2, \ldots, x_k),
\end{split}
\end{align}
where the last equality follows from \cite[Lemma 24 (i)]{PoSuToZa19}. On the other hand, according to \cite[Lemma 3.2 (iii)]{PoSkSu13},
\begin{align}\label{step5PF}
T_{\psi_2 \phi_N}^{b, a, a_1, \ldots, a_k}(p, x_1, \ldots, x_k) = b_Npa_N T_{f^{[k]}}^{a, a_1, \ldots, a_k}(x_1, x_2, \ldots, x_k).
\end{align}
We apply the latter identities to $p=p_{\alpha}$, where $(p_{\alpha})_{\alpha} \subset \mathcal{M}$ is an increasing family of $\tau$-finite projections such that $p_{\alpha} \uparrow 1$ and $\|ap_{\alpha} - p_{\alpha}a\|_{p_0} \leq 1$ for all $\alpha$. Such a sequence exists, see \cite[Proposition 6.6]{PaWiSu02}. Hence, it follows from \eqref{step1PF}, \eqref{step3PF}, \eqref{step4PF} and \eqref{step5PF} that for every $\alpha$,
\begin{align}\label{step6PF}
\begin{split}
&T_{f^{[k+1]}}^{b, a, a_1, \ldots, a_k}(bb_Np_{\alpha}a_N - b_Np_{\alpha}a_Na, x_1, \ldots, x_k)\\
& \ =T_{f^{[k]}}^{b, a_1, \ldots, a_k}(b_Np_{\alpha}a_Nx_1, x_2, \ldots, x_k) - b_Np_{\alpha}a_N T_{f^{[k]}}^{a, a_1, \ldots, a_k}(x_1, x_2, \ldots, x_k).
\end{split}
\end{align}
For the rest of the proof, let us denote $\frac{1}{p} := \sum_{j=0}^k \frac{1}{p_j}$ and $\frac{1}{q} := \sum_{j=1}^k \frac{1}{p_j}$ and notice that $1<p<\infty$ and $1<q<\infty$. As in the proofs of \cite[Lemma 5.2 and Lemma 7.3]{PaWiSu02}, the net $(b b_Np_{\alpha}a_N - b_Np_{\alpha}a_Na)_{\alpha}$ converges weakly in $L_{p_0}(\mathcal{M})$ to $bb_Na_N - b_Na_Na$. This implies, according to Lemma \ref{MOIweakstarcont}, that
\begin{align}\label{step7PF}
T_{f^{[k+1]}}^{b, a, a_1, \ldots, a_k}(bb_Np_{\alpha}a_N - b_Np_{\alpha}a_Na, x_1, \ldots, x_k) \to T_{f^{[k+1]}}^{b, a, a_1, \ldots, a_k}(bb_Na_N - b_Na_Na, x_1, \ldots, x_k)
\end{align}
weakly in $L_{p}(\mathcal{M})$. On the other hand, since $p_{\alpha} \uparrow 1$ and $x_1 \in L_{p_1}(\mathcal{M})$, $b_Np_{\alpha}a_Nx_1 \to b_Na_Nx_1$ in $L_{p_1}(\mathcal{M})$ by \cite[Lemma 5.1]{PaWiSu02}. Then, using the continuity of $T_{f^{[k]}}^{b, a_1, \ldots, a_k}$ on $L_{p_1}(\cM) \times \cdots \times L_{p_k}(\cM)$ given by Theorem \ref{bddMOI}, we get
\begin{align}\label{step8PF}
T_{f^{[k]}}^{b, a_1, \ldots, a_k}(b_Np_{\alpha}a_Nx_1, x_2, \ldots, x_k) \to T_{f^{[k]}}^{b, a_1, \ldots, a_k}(b_Na_Nx_1, x_2, \ldots, x_k)
\end{align}
in $L_{q}(\mathcal{M})$. By a similar reasoning (using the fact that $a_N T_{f^{[k]}}^{a, a_1, \ldots, a_k}(x_1, x_2, \ldots, x_k) \in L_{q}(\mathcal{M})$),
\begin{align}\label{step9PF}
b_Np_{\alpha}a_N T_{f^{[k]}}^{a, a_1, \ldots, a_k}(x_1, x_2, \ldots, x_k) \to b_Na_N T_{f^{[k]}}^{a, a_1, \ldots, a_k}(x_1, x_2, \ldots, x_k)
\end{align}
in $L_{q}(\mathcal{M})$. Hence, \eqref{step6PF}, \eqref{step7PF}, \eqref{step8PF} and \eqref{step9PF} give
\begin{align}\label{step10PF}
\begin{split}
&T_{f^{[k+1]}}^{b, a, a_1, \ldots, a_k}(bb_Na_N - b_Na_Na, x_1, \ldots, x_k)\\
& \ =T_{f^{[k]}}^{b, a_1, \ldots, a_k}(b_Na_Nx_1, x_2, \ldots, x_k) - b_Na_N T_{f^{[k]}}^{a, a_1, \ldots, a_k}(x_1, x_2, \ldots, x_k).
\end{split}
\end{align}
To finish the proof, note that $b_N, a_N \uparrow 1$ as $N\to +\infty$, so 
$$
bb_Na_N - b_Na_Na = b_N(b-a)a_N \to b-a
$$
in $L_{p_0}(\mathcal{M})$. Similarly, $b_Na_Nx_1 \to x_1$ in  $L_{p_1}(\mathcal{M})$. Hence, by boundedness of MOI, taking the limit as $N \to +\infty$ in \eqref{step10PF} gives the desired formula.
\end{proof}
		
From now on, we adopt the following notation: if $x$ is an operator affiliated with $\mathcal{M}$, then for any $k\in\N$, we denote by $(x)^{k}$ the tuple consisting of $k$ copies of $x$, that is, $(x)^k=\underbrace{x,x,\ldots, x}_{k \text{ copies}}$.
		
We now record a direct consequence of Theorem \ref{Prtb-Thm}.

\begin{crl}\label{Perturbationformulageneral}
Let $n\in\N$ and $1<p<\infty$. Let $a,b\in \widetilde{\cM}$ be self-adjoint operators such that $b\in L_{(n-1)p}(\mathcal{M})\cap L_{np}(\mathcal{M})$. Let $f\in C^{n}(\R)$ and assume that $f^{(n-1)},f^{(n)}$ are bounded. Then, for every $1\le j\le n$ and every $t\in\R$,
$$
T_{f^{[n-1]}}^{(a+tb)^j, (a)^{n-j}}((b)^{n-1}) - T_{f^{[n-1]}}^{(a)^n}((b)^{n-1}) = t \sum_{l=1}^{j} T_{f^{[n]}}^{(a+tb)^{l}, (a)^{n+1-l}}((b)^{n}).
$$
\end{crl}

\begin{lma}\label{resolventstrongCV}
Let $a$ be a self-adjoint and densely defined operator on $\mathcal{H}$ and $b=b^*\in \mathcal{M}$. Then, for every $h\in C_b(\mathbb{R})$,
$$
h(a+tb) \underset{t\to 0}{\to} h(a) \quad \text{strongly.}
$$
\end{lma}
		
\begin{proof}
According to \cite[Theorem VIII.20]{ReSi_book}, it is sufficient to prove that
\begin{align*}
(a+tb-i)^{-1} \underset{t\to 0}{\to} (a-i)^{-1} \quad \text{strongly.}
\end{align*}
This simply follows from the second resolvent identity. Indeed, we have the following equality on $\mathcal{H}$:
$$
(a+tb-i)^{-1}-(a-i)^{-1}=-t(a+tb-i)^{-1}b(a-i)^{-1},
$$
so that
$$
\|(a+tb-i)^{-1}-(a-i)^{-1}\|\le |t|\|b\|.
$$
This yields the convergence in norm of $(a+tb-i)^{-1} - (a-i)^{-1}$ to $0$, whence the strong convergence to $0$.
\end{proof}
		
Fix $1\le j\le n+1$ and consider the function
\begin{align}\label{MOI-map}
\psi_{f}:t\in\R\mapsto T_{f^{[n]}}^{(a+tb)^{j}, (a)^{n+1-j}}((b)^{n}) \in L_p(\mathcal{M}),
\end{align}
where $a,b\in\widetilde{\cM}$ are self-adjoint operators, $b\in L_{np}(\mathcal{M})$, and $f\in C^{n}(\R)$ with $f^{(n)}$ bounded. The following result, which characterizes when the map $\psi_{f}$ is continuous, will be used to prove our main result (Theorem \ref{DiffLp}).

\begin{ppsn}\label{continuityat0}
Let $n\in\N$ and $1<p<\infty$. Let $a,b\in\widetilde{\cM}$ be self-adjoint operators, and let $f\in C^{n}(\mathbb{R})$. Suppose that one of the following assumptions holds:
\begin{enumerate}[\normalfont(i)]
\item $f^{(n)} \in C_0(\mathbb{R})$ and $b \in L_{np}(\mathcal{M})$;\\
or
\item $f\in C^n_b(\R)$ (that is, $f', \ldots, f^{(n)}$ are bounded continuous functions) with $f^{(n)}$ uniformly continuous on $\R$ and $b\in L_{np}(\mathcal{M}) \cap L_{(n+1)p}(\mathcal{M})$;\\
or
\item $f\in C^n_b(\R)$ and $b \in L_{np}(\mathcal{M}) \cap \mathcal{M}$.
\end{enumerate}
Then $\psi_f$ in \eqref{MOI-map} is continuous at $0$.
\end{ppsn}
		
\begin{proof}
(i) Assume that $f^{(n)} \in C_0(\mathbb{R})$ and $b \in L_{np}(\mathcal{M})$.\\
Let $\epsilon > 0$. By approximation, there exists $g \in C_c^{(n+1)}(\mathbb{R})$ such that $\| f^{(n)} - g^{(n)} \|_{\infty} \leq \epsilon$. The proof then follows along the same lines as that of \cite[Proposition 4.7]{ChenHong26}.

\medskip

(ii) First, if $f\in C^{n+1}_b(\R)$, then since $f^{(n)}$ and $f^{(n+1)}$ are bounded, we have, according to Corollary \ref{Perturbationformulageneral},
\begin{align*}
\psi_f(t)-\psi_f(0)=t\sum_{l=1}^{j} T_{f^{[n+1]}}^{(a+tb)^{l}, (a)^{n+2-l}}((b)^{n+1}),
\end{align*}
so that, according to Theorem \ref{bddMOI},
\begin{align*}
\|\psi_f(t)-\psi_f(0) \|_p \leq |t| j C \|f^{(n+1)}\|_{\infty} \|b\|_{(n+1)p}^{n+1}
\end{align*}
for some constant $C>0$. This establishes the continuity of $\psi_f$ at $0$ for such $f$.

Next, let $f\in C^n_b(\R)$ with $f^{(n)}$ uniformly continuous on $\mathbb{R}$ and $\epsilon > 0$. By \cite[Lemma 2.1]{CLMMcd}, there exists $h\in C^{n+1}_b(\R)$ such that $\|f^{(n)} - h^{(n)}\|_{\infty} \leq \epsilon$. In particular, $\psi_h$ is continuous at $0$. Write, for every $t\in \mathbb{R}$,
\begin{align*}
\psi_f(t) - \psi_f(0)
& = (\psi_f(t) - \psi_h(t)) + (\psi_h(t) - \psi_h(0)) + (\psi_h(0) - \psi_f(0)) \\
& = \psi_{f-h}(t) + (\psi_h(t) - \psi_h(0)) + \psi_{h-f}(0).
\end{align*}
According to Theorem \ref{bddMOI}, there exists a constant $C'>0$ such that, for every $t$,
\begin{align*}
\| \psi_f(t) - \psi_f(0) \|_p
& \leq 2C' \| f^{(n)} - h^{(n)} \|_{\infty} + \|\psi_h(t) - \psi_h(0)\|_p\\
& \leq 2C'\epsilon + \|\psi_h(t) - \psi_h(0)\|_p,
\end{align*}
whence the continuity of $\psi_f$ at $0$.
			
\medskip
			
(iii) Assume now that $f',\ldots, f^{(n)}$ are bounded and $b \in L_{np}(\mathcal{M}) \cap \mathcal{M}$.\\
First of all, note that the assumption on $b$ implies in particular that $b \in L_{r}(\mathcal{M})$ for every $np\leq r \leq \infty$. The strategy of the proof is to reduce to the first one. Let $(h_k)_{k\geq 1}$ be a sequence of $C^{\infty}_c(\mathbb{R})$ satisfying:
$$
\forall\, k\in\N, \quad 0\leq h_k\leq 1,
\qquad\hbox{and}\qquad\forall\, s\in \mathbb{R},\quad
h_k(s)\underset{k\to\infty}{\longrightarrow} 1,
$$
and define, for every $(\lambda_0,\ldots,\lambda_n)\in \mathbb{R}^{n+1}$,
$$
H_k(\lambda_0, \ldots, \lambda_n) = h_k(\lambda_0)h_k(\lambda_n).
$$
For $t\in \mathbb{R}$, define
$$\psi^k_f(t) := T_{f^{[n]} H_k}^{(a+tb)^{j}, (a)^{n+1-j}}((b)^{n}).$$
According to \cite[Lemma 3.2]{PoSkSu13} and the definition of MOI, we have, for every $t\in \mathbb{R}$,
$$
\psi^k_f(t) = T_{f^{[n]}}^{(a+tb)^{j}, (a)^{n+1-j}}(h_k(a+tb)b, (b)^{n-2}, bh_k(a)) \quad \text{if} \ 1\leq j \leq n,
$$
and
$$
\psi^k_f(t) =  T_{f^{[n]}}^{(a+tb)^{j}, (a)^{n+1-j}}(h_k(a+tb)b, (b)^{n-2}, bh_k(a+tb)) \quad \text{if} \ j=n+1.
$$
Note that, according to Lemma \ref{resolventstrongCV},
\begin{align}\label{proofctny1}
h_k(a+tb) \to h_k(a) \quad \text{and} \quad (fh_k)(a+tb) \to (fh_k)(a) \quad \text{strongly as} \ t\to 0,
\end{align}
and since $(h_k(a))_k$ is bounded in $\mathcal{M}$ and converges strongly to $1$ as $k\to +\infty$, we have, according to \cite[Corollary 5.5.14 and Theorem 6.2.1]{Sukochev-book},
\begin{align}\label{proofctny2}
h_k(a)b \to b \quad \text{and} \quad bh_k(a) \to b \quad \text{in} \ L_{np} \ \text{as} \ k\to +\infty.
\end{align}
From now on, assume that $j=n+1$ to simplify the notations, the others cases being similar. Let $\epsilon > 0$. According to \eqref{proofctny1}, \eqref{proofctny2} and Theorem \ref{bddMOI}, for a fixed $k$ large enough and $|t|$ small enough, we have
\begin{align*}
& \| \psi_f(t) - \psi_f(0) \|_p \\
& \leq \| \psi_f(t) - \psi^k_f(t) \|_p + \|\psi^k_f(t) - \psi^k_f(0) \|_p + \|\psi^k_f(0) - \psi_f(0)\|_p \\
& \leq \|\psi^k_f(t) - \psi^k_f(0) \|_p + C \|f^{(n)}\|_{\infty} \|b\|_{np}^{n-1} \max \{ \|h_k(a+tb)b-b\|_{np}, \|bh_k(a+tb)-b\|_{np},  \\ 
& \hspace{9cm}  \|h_k(a)b-b\|_{np},  \|bh_k(a)-b\|_{np} \} \\
& \leq \|\psi^k_f(t) - \psi^k_f(0) \|_p + C \|f^{(n)}\|_{\infty} \|b\|_{np}^{n-1} \epsilon.
\end{align*}
Hence, it is sufficient to prove the continuity of $\psi^k_f$ at $0$. We prove it by induction on $n$ and with $j=n+1$. To take the integer $n$ into account in our notations, let us write, for $t\in \mathbb{R}$,
$$
\psi_{f,n}(t) = T_{f^{[n]}}^{(a+tb)^{n+1}}((b)^{n}) \quad \text{and} \quad \psi^k_{f,n}(t) =T_{f^{[n]} H_k}^{(a+tb)^{n+1}}((b)^{n}).
$$
For $n=1$, one easily checks that
$$
f^{[1]}H_k(\lambda_0,\lambda_1) = (fh_k)^{[1]}(\lambda_0,\lambda_1)h_k(\lambda_1) - h_k^{[1]}(\lambda_0,\lambda_1) (fh_k)(\lambda_1),
$$
so that
$$
\psi^k_{f,1}(t) = T_{(fh_k)^{[1]}}^{a+tb, a+tb}(b h_k(a+tb)) - T_{h_k^{[1]}}^{a+tb, a+tb}(b(fh_k)(a+tb)).
$$
Since $h_k, fh_k \in C^1_c(\mathbb{R})$, by \eqref{proofctny1} and the case (i) of this proposition, we obtain the continuity of $\psi^k_{f,1}$ at $0$. Next, assume that for some $n\in \mathbb{N}$, $\psi^k_{f,l}$ is continuous at $0$ for every $1\leq l \leq n-1$. As discussed above, this implies in particular that $\psi_{f, l}$ is continuous for every $1\leq  l \leq n-1$. According to \cite[Eq. (3.21)]{Co22},
\begin{align*}
f^{[n]}H_k(\lambda_0, \ldots, \lambda_{n}) 
& = (f h_k)^{[n]}(\lambda_0, \ldots, \lambda_{n})h_{k}(\lambda_{n})- (h_k)^{[n]}(\lambda_0, \ldots, \lambda_{n}) (f h_k)(\lambda_{n})   \\
& \hspace{0,5cm} - h_{k}(\lambda_{n})\sum_{l=1}^{n-1} (h_k)^{[l]}(\lambda_0, \ldots, \lambda_l)f^{[n-l]}(\lambda_l, \ldots, \lambda_{n}).
\end{align*}
Hence, as before,
\begin{align*}
\psi^k_{f,n}(t)
& = T_{(fh_k)^{[n]}}^{(a+tb)^{n+1}}((b)^{n-1},bh_{k}(a+tb)) - T_{h_k^{[n]}}^{(a+tb)^{n+1}}((b)^{n-1},b(fh_k)(a+tb)) \\
& \hspace{0,5cm} - \sum_{l=1}^{n-1} T_{h_k^{[l]}}^{(a+tb)^{l+1}}((b)^{l})\cdot T_{f^{[n-l]}}^{(a+tb)^{n+1-l}}((b)^{n-l-1}, bh_{k}(a+tb)) \\
& = T_{(fh_k)^{[n]}}^{(a+tb)^{n+1}}((b)^{n-1}, bh_{k}(a+tb)) - T_{h_k^{[n]}}^{(a+tb)^{n+1}}((b)^{n-1},b(fh_k)(a+tb)) \\
& \hspace{0,5cm}-\sum_{l=1}^{n-1} T_{h_k^{[l]}}^{(a+tb)^{l+1}}((b)^{l}) \left[\psi_{f,n-l}(t)+T_{f^{[n-l]}}^{(a+tb)^{n+1-l}}((b)^{n-l-1},bh_{k}(a+tb)-b)\right].
\end{align*}
Note that all these operators are well defined using the fact that $b \in L_{r}(\mathcal{M})$ for every $np\leq r \leq \infty$. Now, with the same arguments as for the case $n=1$, $\psi^k_{f,n}$ is continuous at $0$ as a sum and composition of continuous functions. This concludes the proof.
\end{proof}
		
The differentiability result we obtain is the following.

\begin{thm}\label{DiffLp}
Let $1<p<\infty$, $n\in\N$, and $f\in C^{n}_b(\R)$. Let $a,b\in\widetilde{\cM}$ be self-adjoint operators with $b\in L_{p}(\mathcal{M}) \cap L_{np}(\mathcal{M})$. Consider the function
\[\phi : t\in\R\mapsto f(a+tb)-f(a) \in L_p(\mathcal{M}).\]
Then the function $\phi$ is $n$-times differentiable on $\mathbb{R}$ and for every integer $1\le k\le n$,
\[\frac{1}{k!}\phi^{(k)}(t)=T_{f^{[k]}}^{a+tb,\ldots,a+tb}(b,\ldots,b).\]
Moreover, $\phi^{(n)}$ is continuous on $\R$ if any of the following holds: $f^{(n)} \in C_0(\R)$; $f^{(n)}$ is uniformly continuous on $\R$ and $b\in L_{p}(\cM)\cap L_{(n+1)p}(\cM)$; or $b\in L_{p}(\cM)\cap\cM$.
\end{thm}
		
\begin{proof}
We will prove the result by induction on $n$ and in both steps, we will make the proof with the additional assumption that $b \in L_{p}(\mathcal{M}) \cap \mathcal{M}$ and then deduce the result for $L_{p}(\mathcal{M}) \cap L_{np}(\mathcal{M})$. Recall that this implies that for every $p \leq r \leq \infty$, $b\in L_r(\mathcal{M})$. Also, by translation, it is sufficient to prove the differentiability at $t=0$.
			
Assume that $n=1$ and $b \in L_{p}(\mathcal{M}) \cap \mathcal{M}$. According to Theorem \ref{Perturbation-Theorem}, we have, for every $t\neq 0$,
\begin{align*}
\frac{\phi(t)-\phi(0)}{t} =  \frac{1}{t} (f(a+tb)-f(a))
& = T^{a+tb,a}_{f^{[1]}}(b),
\end{align*}
which, according to Proposition \ref{continuityat0}, converges to $T^{a,a}_{f^{[1]}}(b)$ as $t\to 0$.
			
Assume now that $b\in L_{p}(\mathcal{M})$. Let $\epsilon > 0$. By density, there exists $\tilde{b}\in L_{p}(\mathcal{M}) \cap \mathcal{M}$ such that $\|b-\tilde{b}\|_p < \epsilon$. Define
\[ \tilde{\phi} : t\in\R\mapsto f(a+t\tilde{b})-f(a).\]
From the above computations, the function $\tilde{\phi}$ is differentiable at $0$ and there exists $\eta > 0$ such that, for every $0<|t| < \eta$,
\begin{equation*}
\left\| \frac{\tilde{\phi}(t)-\tilde{\phi}(0)}{t}-T^{a,a}_{f^{[1]}}(\tilde{b}) \right\|_p < \epsilon.
\end{equation*}
Hence, for all $0<|t| < \eta$, we have, according to Theorem \ref{Perturbation-Theorem}, Theorem \ref{bddMOI} and by triangle inequality,
\begin{align*}
& \left\| \frac{\phi(t) - \phi(0)}{t} - T^{a,a}_{f^{[1]}}(b) \right\|_p \\
& \hspace*{0,5cm} \leq \left\| \frac{\phi(t) - \tilde{\phi}(t)}{t} \right\|_p + \left\| T^{a,a}_{f^{[1]}}(\tilde{b}) - T^{a,a}_{f^{[1]}}(b) \right\|_p + \left\| \frac{\tilde{\phi}(t) - \tilde{\phi}(0)}{t} - T^{a,a}_{f^{[1]}}(\tilde{b}) \right\|_p \\
& \hspace*{0,5cm} = \left\| T^{a+tb,a+t\tilde{b}}_{f^{[1]}}(b-\tilde{b}) \right\|_p + \left\| T^{a,a}_{f^{[1]}}(\tilde{b}-b) \right\|_p + \left\| \frac{\tilde{\phi}(t) - \tilde{\phi}(0)}{t} - T^{a,a}_{f^{[1]}}(\tilde{b}) \right\|_p \\
& \hspace*{0,5cm} \leq 2C\|f'\|_{\infty} \|b-\tilde{b} \|_p + \epsilon = \text{const} \cdot \epsilon,
\end{align*}
which yields the differentiability of $\phi$ at $0$ with the desired derivative.
			
Assume now that $\phi$ is $(n-1)$-times differentiable and that for every $1\leq k \leq n-1$ and every $t\in \mathbb{R}$,
\[\frac{1}{k!}\phi^{(k)}(t)=T_{f^{[k]}}^{a+tb,\ldots,a+tb}(b,\ldots,b).\]
Assume that $b \in L_{p}(\mathcal{M}) \cap \mathcal{M}$. According to Corollary \ref{Perturbationformulageneral}, for every $t\in \mathbb{R}$,
$$
\dfrac{\phi^{(n-1)}(t) - \phi^{(n-1)}(0)}{t(n-1)!} = \sum_{l=1}^{n} T_{f^{[n]}}^{(a+tb)^{l}, (a)^{n+1-l}}((b)^{n}),
$$
which converges to
$$
n T_{f^{[n]}}^{(a)^{n+1}}((b)^{n})
$$
as $t\to 0$, according to Proposition \ref{continuityat0}. This proves the result for such $b$. To deduce the general case, let $b\in L_{p}(\mathcal{M}) \cap L_{np}(\mathcal{M})$ and approximate it by an element $\tilde{b} \in L_{p}(\mathcal{M}) \cap \mathcal{M}$ such that $\|b-\tilde{b}\|_{np}$ is small and then follow the same proof as in \cite[Lemma 3.7]{Co22}. Continuity of $\phi^{(n)}$ follows from Proposition \ref{continuityat0}.
\end{proof}

\begin{rmrk}
Note that the $n=1$ case of the preceding theorem first appeared in \cite[Corollary 6.10]{PaSu04}, where it was proved for $f\in C^{1}(\R)$ such that $f'$ has a bounded variation. Theorem \ref{DiffLp} extends the function class to $f\in C^{1}(\R)$ with bounded $f'$. 
\end{rmrk}

\begin{rmrk}
Theorem \ref{DiffLp} remains valid if we assume only that $a$ is affiliated with $\mathcal{M}$ (instead of being $\tau$-measurable) and that $b\in L_{p}(\mathcal{M})\cap\mathcal{M}$.
\end{rmrk}
		
The assumption $b\in L_p(\mathcal{M}) \cap L_{np}(\mathcal{M})$ in Theorem \ref{DiffLp} cannot be weakened to $b\in L_p(\mathcal{M})$ and justified in the example below.

\begin{xmpl}  
Let $f : \mathbb{R} \to \mathbb{R}$ be a $C^2$-function such that $f',f''$ are bounded and such that $f(x) = \frac{1}{1+x}$ for $x\geq 0$. Consider the (commutative) $L_p$-space $L_p(0,1)$ for some $1<p<\infty$. Let $a$ be the constant function $1$ and let $b$ be a positive element of $L_p(0,1) \setminus L_{2p}(0,1)$. Then the function $f^{[1]}$ is given, for $x,y\geq 0$, by $f^{[1]}(x,y) = - \frac{1}{(1+x)(1+y)}$. The function $\phi : \mathbb{R} \ni t \mapsto f(a+tb)-f(a) \in L_p(0,1)$ is differentiable with derivative given by
$$\forall t\geq 0, \quad \phi'(t) = T_{f^{[1]}}^{a+tb, a+tb}(b) = - \frac{b}{(2+tb)^2} \in L_p(0,1).
$$
If $\phi'$ was differentiable at $t=0$, then
\begin{align*}
\frac{\phi'(t)-\phi'(0)}{t} = \frac{b^2}{(2+tb)^2} + t \frac{b^3}{4(2+tb)^2}
\end{align*}
would have a limit $g \in L_p(0,1)$ as $t\to 0^+$. In that case, we could find a sequence $(t_n)_n$ converging to $0$ and such that
$$
\frac{b^2}{(2+t_nb)^2} + t_n \frac{b^3}{4(2+t_nb)^2} \underset{n\to +\infty}{\longrightarrow} g
$$
almost everywhere. But the latter sequence converges almost everywhere to $\frac{b^2}{4}$ which implies that $g=\frac{b^2}{4}$. This is a contradiction since $b^2 \notin L_{p}(0,1)$, hence $\phi'$ is not differentiable at $t=0$.
\end{xmpl}

\section{Trace formula: noncommutative $L_p$-perturbation}\label{sec:traceLp}
In this section, we establish the $n$-th order ($n \ge 2$) trace formula \eqref{PSS-formula-trace} for perturbations $b$ in $L_n(\cM) \cap L_{n^{2}}(\cM)$ and for functions $f$ from a significantly larger class introduced below.

\smallskip

We now introduce a new class of scalar functions on $\R$ for which we establish the trace formula.

\noindent {\bf The class $\mathfrak{W}_n(\R)$:} Let $n\in \N$. Define
\begin{align}\label{Funct_Class-1}
\mathfrak{W}_n(\R)=\{f\in C^n_0(\R): f^{[n]}\in\mathfrak{C}_n \text{ and } f^{(k)}\in C_b(\R), 1\leq k\leq n-1\}.
\end{align}
The following lemma ensures that the newly introduced function class $\mathfrak{W}_{n}(\R)$ is sufficiently large.

\begin{lma}\label{lem_inclussion0}
Let $n\in\N$. Then
\smallskip
\begin{enumerate}[\normalfont(i)]
\item $C_c^{n+1}(\R)\subseteq\mathcal{W}_n(\R)\subseteq\mathfrak{W}_n(\R).$
\item $B_{\infty 1}^n(\R)\cap C^n_0(\R) \cap C^{n-1}_b(\R)\subseteq \mathfrak{W}_n(\R).$
\end{enumerate}
\end{lma}

\begin{proof}
The set inclusions follow from Lemma \ref{setinclussion}.
\end{proof}

\begin{thm}\label{main_thm_Lp}
Let $n\in\N$ and $n\geq 2$. Let $a, b\in \widetilde{\cM}$ be self-adjoint operators such that $b\in L_{n}(\cM)\cap L_{n^2}(\cM)$. Then there exists a unique function $\eta_{n}:=\eta_{n,a,b}\in L^1(\R)$ such that 
\begin{align}\label{formula_1}
\tau\left(\cR_n(f,a,b)\right)= \int_\R f^{(n)}(t)\eta_n(t) dt
\end{align}
for every $f\in\fW_n(\R)$. Moreover, $\eta_n$ satisfies
\begin{align*}
\| \eta_n \|_1 \leq c_n \|b\|_n^n.
\end{align*}
\end{thm}

\begin{proof}
Let $b\in L_n(\cM)\cap L_{n^2}(\cM)$, and let $f\in \fW_n(\R)$. By Proposition~\ref{continuityat0}, the map
\[
[0,1]\ni t \longmapsto g(t) := n(1-t)^{n-1}\, T_{f^{[n]}}^{(a+tb)^{n+1}}\big((b)^n\big) \in L_{n}(\cM)
\]
is continuous. Therefore, by \cite[Theorem 1.1.20 and Proposition 1.2.2]{Neerven}, $g$ is Bochner integrable as an $L_n(\cM)$-valued function.

Applying Theorems \ref{DiffLp}, \ref{Perturbation-Theorem}, \ref{Prtb-Thm}, and Lemma~\ref{nice_moi_bdd}, we obtain
\begin{align}\label{eq:remainder-traceclass}
\cR_n(f,a,b)=\int_0^1 g(t)\,dt=T_{f^{[n]}}^{a+b,a,\ldots,a}\big((b)^n\big) \in L_1(\cM),
\end{align}
where the integration in the first equality is the $L_n(\cM)$-valued Bochner integral.

Next, we want to show that
\begin{align}\label{eq:trace-inside}
\tau\!\Big(\int_0^1 g(t)\,dt\Big) = \int_0^1 \, \tau\!\big(g(t)\big)\,dt.
\end{align}

To this end, let $e\in \cM$ be a $\tau$-finite projection and set $\tilde g(t) := g(t)e$. By continuity of $g$ and H\"older's inequality,
\[\tilde g : [0,1] \to L^1(\cM)\]
is continuous. By \cite[Theorem 1.1.20 and Proposition 1.2.2]{Neerven}, $\tilde g$ is Bochner integrable as an $L_1(\cM)$-valued function.
	
Hence, by continuity of $\tau$, 
	
\[\tau\left(\underbrace{\left( \int_{0}^{1} g(t)\,dt \right)}_{\text{as an element of} \ L_n(\cM)}e\right) = \tau \left( \int_0^1 g(t)e\,dt \right) = \int_0^1 \tau(g(t)e)\,dt.\]

Let $(e_n)_n$ be a sequence of $\tau$-finite projections with $e_n \uparrow 1$, which exists by \cite[Proposition~5.6.25]{Sukochev-book}. Note that by \eqref{eq:remainder-traceclass}, $\int_{0}^{1}g(t)\, dt\in L_1(\cM)$. Then, according to \cite[Corollary 5.5.14 and Theorem 6.2.1]{Sukochev-book}, 
\[
\Big(\int_0^1 g(t)\,dt\Big) e_n \longrightarrow \int_0^1 g(t)\,dt \quad \text{in } L_1(\cM),
\]
and thus
\[
\tau\Big(\big(\int_0^1 g(t)\,dt\big)e_n\Big) \longrightarrow \tau\Big(\int_0^1 g(t)\,dt\Big).
\]
On the other hand, Lebesgue’s dominated convergence theorem gives
\[
\int_0^1 \tau\big(g(t)e_n\big)\,dt \longrightarrow \int_0^1 \tau\big(g(t)\big)\,dt.
\]
This proves \eqref{eq:trace-inside}.

Thus, by Corollary \ref{cr:trace-bdd}, we obtain
\begin{align}\label{bound_2}
\big|\tau(\cR_n(f,a,b))\big| = \Big|\tau\Big(\int_0^1 g(t)\,dt\Big)\Big| \le \int_0^1 \big|\tau(g(t))\big|\,dt \le c_n \,\|f^{(n)}\|_\infty \,\|b\|_n^n.
\end{align}

Consider now the linear functional
\[
\{f^{(n)}: f\in \fW_n(\R)\} \ni f^{(n)} \longmapsto  \tau\!\Big(\cR_n(f, a, b)\Big),
\]
which is bounded by \eqref{bound_2}. By the Riesz representation theorem for functionals in $(C_0(\R))^*$, the Hahn-Banach theorem, and \eqref{bound_2}, there exists a finite complex measure $\mu_n$ such that
\begin{align}\label{measure exist}
\tau(\cR_n(f,a,b)) = \int_{\R} f^{(n)}(t)\,d\mu_n(t)
\end{align}
for every $f\in \fW_n(\R)$. In particular, \eqref{measure exist} holds for all $f\in C_c^{n+1}(\R)$.

The remainder of the proof, namely the verification that $\mu_n$ is absolutely continuous with respect to the Lebesgue measure, follows verbatim the argument in the proof of \cite[Theorem~1.1]{PoSkSu13}. It relies on the estimates established in Theorem~\ref{bddMOI} and Corollary~\ref{cr:trace-bdd}, together with density of $L_1(\cM)\cap \cM$ in $L_n(\cM)\cap L_{n^2}(\cM)$, the perturbation formulas from Theorems \ref{Perturbation-Theorem} and~\ref{Prtb-Thm}, and the derivative formula of Theorem~\ref{DiffLp}. The proof is complete.
\end{proof}

\begin{rmrk}
Using the technique developed in the current section, one can establish \eqref{formula_1} for $n=1$, with $b\in L_1(\cM)\cap L_2(\cM)$ and $f\in \fW_1(\R)$. This, however, does not lead to any substantial improvement over the formula already obtained in \cite{PoSu14jst}.
\end{rmrk}

\begin{rmrk}
Theorem \ref{main_thm_Lp} remains valid if we assume only that $a$ is affiliated with $\mathcal{M}$ (instead of being $\tau$-measurable) and that $b\in L_{n}(\cM)\cap\mathcal{M}$.
\end{rmrk}

\section{Trace formula: Schatten $p$-class perturbation}\label{sec:traceSp}

In this section, we assume that $\mathcal{M}=\bh$ and that $b=b^*\in\mathcal{S}^n(\hilh)$. The aim of this section is to establish a trace formula for the Taylor remainder $\cR_n(f,a,b)$ for a class of functions containing $\fW_n(\R)$, thereby extending the corresponding result of \cite{PoSkSu13}, and to derive a trace formula for an object analogous to $\cR_n(f,a,b)$ that is valid for a broader class of functions, including the Besov space $B_{\infty 1}^n(\R)$, thus extending the result of \cite{AlPe11}.

These extensions are made possible by the explicit nature of the canonical trace $\Tr$, which allows us to employ the theory of a substantially stronger notion of multiple operator integrals developed in \cite{CoLeSu}. This framework applies to a larger class of symbols and yields estimates of the same type as those available for the previously introduced notions.

\medskip

\noindent{\bf Multiple operator integral: Coine-Le Merdy-Sukochev's approach.} 
The following notion of multiple operator integration was developed in \cite{CoLeSu}.

\medskip

Let $n\in\N$ with $n\geq 2$, and let $E_{1},\ldots,E_{n},E$ be Banach spaces. We denote by $\mathcal{B}_{n}(E_{1}\times\cdots\times E_{n},E)$ the space of $n$-linear mappings $T:E_{1}\times\cdots\times E_{n}\to E$ such that
\[\|T\|_{\mathcal{B}_{n}(E_{1}\times\cdots\times E_{n},E)}:=\sup_{\|x_{i}\|\le 1,\, 1\le i\le n}\|T(x_{1},\ldots,x_{n})\|<\infty.\]
We simply write $\mathcal{B}_{n}(E)$ in the case where $E_{1}=\cdots=E_{n}=E$.

Consider self-adjoint operators $a_1, \dots, a_n$ in $\hilh$, each associated with a scalar-valued spectral measure $\lambda_{a_1}, \dots, \lambda_{a_n}$ (see \cite{conway}). Define the multilinear transformation

\begin{align*}
\Gamma^{a_1,\dots,a_n}:L^\infty(\lambda_{a_1}) \otimes \cdots \otimes L^\infty(\lambda_{a_n}) \to \mathcal{B}_{n-1}(\Sp^2(\hilh)),
\end{align*}
as the unique linear map such that for any $f_i \in L^\infty(\lambda_{a_i}),~i=1,\ldots,n$, and for any $b_1,\dots,b_{n-1}\in\mathcal{S}^2(\hilh)$, we have
\begin{align*}
\left[\Gamma^{a_1,\dots,a_n}(f_{1}\otimes\cdots\otimes f_{n})\right](b_1, \dots, b_{n-1})
= f_1(a_1)b_1 f_2(a_2) \cdots f_{n-1}(a_{n-1}) b_{n-1} f_n(a_n).
\end{align*}

Note that $\mathcal{B}_{n-1}(\mathcal{S}^2(\hilh))$ is a dual space. According to \cite[Proposition 3.4]{CoLeSu}, $\Gamma^{a_1,\dots,a_n}$ extends to a unique $w^*$-continuous isometry still denoted by
\begin{align*}
\Gamma^{a_1,\dots,a_n}: L^\infty\left(\prod_{i=1}^n \lambda_{a_{i}}\right) \to \mathcal{B}_{n-1}(\mathcal{S}^2(\hilh)).
\end{align*}
See \cite[Section 3]{CoLeSu} for more information in the case $n=3$.

\begin{dfn}\label{moi_coine}
For $\varphi \in L^\infty\left( \prod_{i=1}^n \lambda_{a_{i}} \right)$, the operator $\Gamma^{a_1, \ldots, a_n}(\varphi)$ is called a multiple operator integral associated to $a_1,\ldots,a_n$ and $\varphi$.
\end{dfn}

The $w^*$-continuity of $\Gamma^{a_1,\dots,a_n}$ means that if a net $(\varphi_i)_{i\in I}$ in $L^\infty\left(\prod_{i=1}^n \lambda_{a_{i}}\right)$ converges to $\varphi\in L^\infty\left(\prod_{i=1}^n \lambda_{a_{i}}\right)$ in the $w^*$-topology, then for any $b_1,\ldots, b_{n-1}\in\mathcal{S}^2(\hilh)$,
\begin{align*}
\left[\Gamma^{a_1,\dots, a_n}(\varphi_i)\right](b_1, \dots, b_{n-1})\xrightarrow[i]{\text{\normalfont weakly in $\mathcal{S}^{2}(\hilh)$}}\left[\Gamma^{a_1,\dots, a_n}(\varphi)\right](b_1, \dots, b_{n-1}).
\end{align*}
Consider $p_1,\dots,p_{n-1},p\in (1,\infty)$ such that $\frac{1}{p}=\frac{1}{p_{1}}+\cdots+\frac{1}{p_{n-1}}$ and suppose $\varphi\in L^\infty\left(\prod_{i=1}^n \lambda_{a_{i}} \right)$. We write
\begin{align*}
\Gamma^{a_1,\ldots,a_n}(\varphi) \in \mathcal{B}_{n-1}(\mathcal{S}^{p_1}(\hilh)\times \cdots\times\mathcal{S}^{p_{n-1}}(\hilh), \mathcal{S}^{p}(\hilh)),
\end{align*}
if $\Gamma^{a_1,\ldots,a_n}(\varphi)$ defines a bounded $(n-1)$-linear mapping
\begin{align*}
\Gamma^{a_1,\ldots,a_n}(\varphi) : (\mathcal{S}^2(\hilh)\cap\mathcal{S}^{p_1}(\hilh))\times\cdots \times(\mathcal{S}^2(\hilh)\cap\mathcal{S}^{p_{n-1}}(\hilh))\to \mathcal{S}^{p}(\hilh),
\end{align*}
where each $\mathcal{S}^2(\hilh) \cap\mathcal{S}^{p_i}(\hilh)$ is equipped with the $\|\cdot\|_{p_i}$-norm. Due to the density of $\mathcal{S}^2(\hilh) \cap\mathcal{S}^{p_i}(\hilh)$ in $\mathcal{S}^{p_i}(\hilh)$, the mapping extends uniquely to a multilinear operator
\begin{align*}
\Gamma^{a_1,\ldots,a_n}(\varphi):\mathcal{S}^{p_1}(\hilh)\times \cdots\times\mathcal{S}^{p_{n-1}}(\hilh)\to\mathcal{S}^{p}(\hilh).
\end{align*}

\begin{lma}{\normalfont\cite[Lemma 2.3]{Co22}}\label{weak_conv}
Let $n\in\N$ and $n\ge2$. Let $1<p,p_{j}<\infty$, $j=1,\ldots,n-1$ be such that $\frac{1}{p}=\frac{1}{p_{1}}+\cdots+\frac{1}{p_{n-1}}$. Let $a_1,\ldots,a_n$ be self-adjoint operators in $\hilh$ and $(\varphi_k)_{k\geq 1},\varphi\in L^\infty(\lambda_{a_1}\times\cdots\times\lambda_{a_n})$. Assume that $(\varphi_k)_{k}$ is $w^*$-convergent to $\varphi$ and that 
$\left(\Gamma^{a_1,\ldots, a_n}(\varphi_k)\right)_{k\geq 1}\subset\mathcal{B}_{n-1}(\mathcal{S}^{p_1}(\hilh)\times\cdots\times\mathcal{S}^{p_{n-1}}(\hilh), \mathcal{S}^p(\hilh))$ is bounded. Then 
\[
\Gamma^{a_1,\ldots,a_n}(\varphi)\in \mathcal{B}_{n-1}(\mathcal{S}^{p_1}(\hilh)\times\cdots\times\mathcal{S}^{p_{n-1}}(\hilh), \mathcal{S}^p(\hilh))
\]
with
\begin{align*}
\left\|\Gamma^{a_1,\ldots,a_n}(\varphi) \right\|_{\mathcal{B}_{n-1}(\mathcal{S}^{p_1}(\hilh)\times\cdots\times\mathcal{S}^{p_{n-1}}(\hilh), \mathcal{S}^p(\hilh))} 
\leq\liminf_{k}\left\| \Gamma^{a_1,\ldots,a_n}(\varphi_k) \right\|_{\mathcal{B}_{n-1}(\mathcal{S}^{p_1}(\hilh)\times\cdots\times\mathcal{S}^{p_{n-1}}(\hilh), \mathcal{S}^p(\hilh))}
\end{align*}
and for any $b_i\in \mathcal{S}^{p_i}(\hilh), 1\leq i\le n-1$,
\begin{align*}
\left[\Gamma^{a_1,\ldots, a_n}(\varphi_k)\right](b_1,\ldots,b_{n-1}) 
\xrightarrow[k \to +\infty]{\text{\normalfont weakly in $\mathcal{S}^{p}(\hilh)$}} 
\left[\Gamma^{a_1,\ldots,a_n}(\varphi)\right](b_1,\ldots,b_{n-1}).
\end{align*}
\end{lma}

\medskip

\noindent We next introduce the following new function class.
\begin{align}\label{Funct_Class-2}
\fQ_n(\R)=\bigl\{f\in C_b^{n-1}(\R)\cap D_b^n(\R):f^{[n]} \in \fA_n\bigr\}.
\end{align}

\medskip

The following result is an immediate consequence of Lemma \ref{lem_inclussion0}.

\begin{lma}\label{lem_inclussion}
Let $n \in \N$. Then
\begin{enumerate}[\normalfont(i)]
\item $\mathcal{W}_n(\R) \subseteq \mathfrak{W}_n(\R) \subseteq \fQ_n(\R)$.
\item $B_{\infty 1}^1(\R)\cap B_{\infty 1}^n(\R)\subseteq B_{\infty 1}^n(\R)\cap C^n_b(\R)\subseteq \fQ_n(\R)$.
\end{enumerate}
\end{lma}

The following lemma is crucial for the proof of our main result.

\begin{lma}\label{moi_eq}	
Let $n\in\N$ with $n\geq 2$, and let $p_1,\ldots, p_n,p,r \in (1,\infty)$ satisfy
$\frac{1}{p}=\sum_{i=1}^{n}\frac{1}{p_i},~~\frac{1}{r}=\sum_{i=1}^{n-1}\frac{1}{p_i}.$ Let $a_1,\ldots,a_{n+1}$ be self-adjoint operators in $\hilh$, and let $b_i\in\mathcal{S}^{p_i}(\hilh)$ for $i=1,\ldots,n$. For $f\in D^n(\R)$, define
\begin{align*}
\phi_{f^{[n]}}(\lambda_0,\ldots,\lambda_{n-1})= f^{[n]}(\lambda_0,\ldots,\lambda_{n-1},\lambda_0).
\end{align*}
Then, for $f\in D^n_b(\R)$ such that $f^{[n]}\in\fA_n$ admits the representation \eqref{function_representation}, we have
\begin{align}
\label{eq:neq-1}
\left[\Gamma^{a_1,\ldots,a_{n+1}}(f^{[n]})\right](b_1,\ldots,b_n)
&=\int_{\Omega}\alpha_0(a_1,\omega)\, b_1\alpha_1(a_2,\omega)\, b_2\cdots b_n\alpha_n(a_{n+1},\omega)\, d\nu(\omega),\\
\label{eq:neq-2}
\left[\Gamma^{a_1,\ldots,a_{n}}(\phi_{f^{[n]}})\right](b_1,\ldots,b_{n-1})
&=
\int_{\Omega}	\alpha_n(a_1,\omega)\alpha_0(a_1,\omega)\, b_1\alpha_1(a_2,\omega)\, b_2\cdots
b_{n-1}\alpha_{n-1}(a_n,\omega)\, d\nu(\omega).
\end{align}
Moreover, for any $f\in D_b^n(\R)$ (that is, $f^{(n)}$ exists and is bounded), there exist constants $c_{p,n},c_{r,n}>0$ such that
\begin{align}
\label{eq:neq-3}
\left\|\left[\Gamma^{a_1,\ldots,a_{n+1}}(f^{[n]})\right](b_1,\ldots,b_n)\right\|_p&\le c_{p,n}\,\|f^{(n)}\|_\infty\|b_1\|_{p_1}\cdots\|b_n\|_{p_n},\\
\label{eq:neq-4}
\left\|\left[\Gamma^{a_1,\ldots,a_{n}}(\phi_{f^{[n]}})\right](b_1,\ldots,b_{n-1})\right\|_r&\le c_{r,n}\,\|f^{(n)}\|_\infty\|b_1\|_{p_1}\cdots\|b_{n-1}\|_{p_{n-1}}.
\end{align}
\end{lma}

\begin{proof}
We prove \eqref{eq:neq-1} and \eqref{eq:neq-3}; the remaining statements \eqref{eq:neq-2} and \eqref{eq:neq-4} follow analogously.
	
The identity \eqref{eq:neq-1} follows from \cite[Proposition~7.4]{CoLeSu} together with the density of $\mathcal{S}^2(\hilh)\cap\mathcal{S}^{p_i}(\hilh)$ in $\mathcal{S}^{p_i}(\hilh)$. The estimate follows from (the proof of) \cite[Theorem 2.7]{Co22}, together with similar density arguments for Schatten classes.
\end{proof}

The following theorem is our main result in this section.

\begin{thm}\label{trace_formula_thm2}
Let $n\in\N$ and $n\geq 2$. Let $a$ be a self-adjoint operator in $\hilh$, and let $b=b^*\in\mathcal{S}^n(\hilh)$.  Then there exists a unique function $\eta_{n}:=\eta_{n,a,b}\in L^1(\R)$ such that
\begin{align}\label{iden_I}
\Tr\left(\left[\Gamma^{a,a+b,a,\ldots,a}(\phi_{f^{[n]}})\right]((b)^{n-1})b\right)= \int_\R f^{(n)}(t) \eta_n(t) dt,
\end{align}
for every $f\in D^n_b(\R)$. Moreover, 
\begin{align}\label{iden_II}
\Tr\left(\left[\Gamma^{a,a+b,a,\ldots,a}(f^{[n]})\right]((b)^{n})\right)=\int_{\R}f^{(n)}(t)\eta_n(t)dt
\end{align}
for every $f\in D^n_b(\R)$ such that $f^{[n]}\in \fA_n$.
Furthermore, 
\begin{align}\label{iden_III}
\Tr\left(\cR_n(f,a,b)\right)= \int_\R f^{(n)}(t)\eta_n(t) dt
\end{align}
for every $f\in\fQ_n(\R)$.
\end{thm}
	
\begin{proof}
Let $f \in C^n(\mathbb{R})$ and suppose $f^{(n)} \in C_0(\mathbb{R})$. Then $\phi_{f^{[n]}}$ is continuous on $\R^{n}$, and using \cite[Lemma 5.1]{PoSkSu13}, there exists a constant $M_{0}>0$ (independent of $f$) such that 
\[\|\phi_{f^{[n]}}\|_{\infty}\le M_{0}\|f^{(n)}\|_{\infty}.\]  
By H\"older's inequality for Schatten class operators together with Lemma \ref{moi_eq}, there exists a constant $c_n>0$ such that
\begin{align*}
\left|\Tr\left(\left[\Gamma^{a,a+b,a,\ldots,a}(\phi_{f^{[n]}})\right]((b)^{n-1})b\right)\right|\leq c_{n}\, \|f^{(n)}\|_{\infty}\|b\|_{n}^{n}.
\end{align*}
Now, applying the Hahn-Banach theorem to the functional 
$$C_0(\R) \supseteq \{f^{(n)} : f \in C_0^{n}(\R)\} \ni f^{(n)}\mapsto\Tr\left(\left[\Gamma^{a,a+b,a,\ldots,a}(\phi_{f^{[n]}})\right]((b)^{n-1})b\right),$$ 
and then using the Riesz representation theorem for elements in $(C_0(\R))^*$, there exists a unique Borel regular complex measure $\mu$ such that
\begin{align}\label{iden_1}
\Tr\left(\left[\Gamma^{a,a+b,a,\ldots,a}(\phi_{f^{[n]}})\right]((b)^{n-1})b\right)=\int_{\R}f^{(n)}(t)d\mu(t)
\end{align}
for every $f\in C^n(\R)$ satisfying $f^{(n)}\in C_0(\R)$.

Next, we show that \eqref{iden_1} holds for every $f\in C^n(\R)$ such that $f^{(n)}\in C_b(\R)$. Let $f\in C^n(\R)$ with $f^{(n)}\in C_b(\R)$. Then there exists a sequence $\{g_k\}_{k\ge1}\subset C^n(\R)$ such that, for each $k\in\N$, one has $g_k^{(n)}\in C_0(\R)$, $\|g_k^{(n)}\|_\infty\leq M_1$ for some $M_1>0$, and $g_k^{(n)}\to f^{(n)}$ pointwise on $\R$ as $k\to+\infty$ (see the proof of \cite[Theorem 5.1]{ChCoGiPr24} for the construction of such a sequence). This further implies that $\phi_{g_k^{[n]}}\to\phi_{f^{[n]}}$ pointwise and that the sequence $\{\phi_{g_k^{[n]}}\}_{k\ge1}$ is uniformly bounded by $M_0M_1$. Therefore, by the dominated convergence theorem and Lemma \ref{weak_conv}, it follows from \eqref{iden_1} that

\begin{align}\label{iden_2}
\nonumber\Tr\left(\left[\Gamma^{a,a+b,a,\ldots,a}(\phi_{f^{[n]}})\right]((b)^{n-1})b\right)&=\lim\limits_{k\to\infty}\Tr\left(\left[\Gamma^{a,a+b,a,\ldots,a}(\phi_{g_{k}^{[n]}})\right]((b)^{n-1})b\right)\\
&=\int_{\R}f^{(n)}(t)d\mu(t).
\end{align}
Hence, \eqref{iden_1} holds for every $f\in C^n(\R)$ such that $f^{(n)}\in C_b(\R)$. Finally, we show that \eqref{iden_1} holds for every $f\in D^n_b(\R)$, that is, for functions whose $n$-th derivative is bounded. In this case, we follow the same strategy and choose a sequence $\{h_k\}_{k\geq 1}\subset C^n(\R)$ such that $ \{h_k^{(n)}\}_{k\geq 1}$ is uniformly bounded and converges pointwise to $f^{(n)}$ on $\R$. A construction of such a sequence $\{h_k\}_{k\ge1}$ can be found in the proof of \cite[Theorem 5.1]{ChCoGiPr24}. By a similar argument, we conclude that \eqref{iden_1} holds for every $f\in D^n_b(\R)$.

\medskip

Note that, $C_c^{n+1}(\R)\subset D^n_b(\R)$. Let $f\in C_c^{n+1}(\R)$ and suppose $f^{[n]}$ has the representation \eqref{function_representation}. Then, by Lemma \ref{moi_eq} and the cyclicity of the trace, we obtain
\begin{align}\label{iden_3}
\nonumber&\Tr\left(\left[\Gamma^{a,a+b,a,\ldots, a}(\phi_{f^{[n]}})\right]((b)^{n-1})b\right)\\
\nonumber&=\Tr\left(\int_{\Omega} \alpha_{n}(a,\omega)\alpha_0(a,\omega)\,b\,\alpha_1(a+b,\omega)\,b\cdots b\,\alpha_{n-1}(a, \omega)\,b\,d\nu(\omega)\right)\\
\nonumber&= \Tr\left(\int_{\Omega} \alpha_0(a,\omega)\,b\,\alpha_1(a+b,\omega)\,b\cdots b\,\alpha_{n-1}(a, \omega)\,b\,\alpha_{n}(a,\omega)d\nu(\omega)\right)\\
&=\Tr\left(\left[\Gamma^{a,a+b,a,\ldots, a}({f^{[n]}})\right]((b)^{n})\right)=\Tr\left(\cR_n(f,a, b)\right)\stackrel{\text{Theorem \ref{main_thm_Lp}}}{=}\int_{\R} f^{(n)}(t)\eta_n(t)\,dt,
\end{align}
where the second last equality follows from \cite[Proposition 3.3]{Co22}.
Therefore, combining \eqref{iden_2} and \eqref{iden_3}, we conclude that
\begin{align*}
\int_\R f^{(n)}(t) \,d\mu(t)=\int_\R f^{(n)}(t) \eta_n(t)\,dt\quad \forall f\in C_c^{n+1}(\R).
\end{align*}
This implies that $d\mu(t)=\eta_n(t)\,dt$, which establishes the formula \eqref{iden_I}. The formulas \eqref{iden_II} and \eqref{iden_III} then follow as consequences of \eqref{iden_I}. This completes the proof of the theorem.
\end{proof}

\begin{rmrk}\label{similar to peller}
We note that for an arbitrary function $f\in B_{\infty1}^n(\R)$, the derivatives
\[
\frac{d^k}{dt^k}f(a+tb)\big|_{t=0}, \qquad k<n,
\]
need not exist, since $f^{(k)}$ is not necessarily bounded for $1\leq k\leq n-1$. For this reason, the authors of \cite{AlPe11} introduced the expression $\mathcal{T}_{a,b}^{(n)}f$ (see \cite[(5.1)]{AlPe11}), which in our notation coincides with $\left[\Gamma^{a+b,a,a,\ldots, a}({f^{[n]}})\right]((b)^{n})$, as a substitute for $\mathcal{R}_n(f,a,b)$, and established a trace formula in \cite[Theorem 7.1]{AlPe11}.

On the other hand, for sufficiently smooth functions, for instance $f\in C_c^{n+1}(\R)$, it follows from \cite[Proposition 3.3]{Co22} that
\[
\mathcal{R}_n(f,a,b)
=\mathcal{T}_{a,b}^{(n)}f
=\left[\Gamma^{a+b,a,\ldots, a}({f^{[n]}})\right]((b)^{n})
=\left[\Gamma^{a,a+b,a,\ldots, a}({f^{[n]}})\right]((b)^{n}).
\]

In summary, in the terminology of \cite[Theorem 4.6]{Pe05}, Theorem \ref{trace_formula_thm2} implies that the map
\begin{align*}
f\mapsto \cR_n(f,a,b)
\end{align*}
extends from $B_{\infty 1}^1(\R)\cap B_{\infty 1}^n(\R)$ to a bounded linear operator from
\begin{align}\label{new-functionclass}
\mathcal{C}_n(\R):=\{f\in D^n_b(\R): f^{[n]}\in \fA_n\}
\end{align}
to $\Sp^1(\hilh)$, and the trace formula \eqref{iden_II} holds for all functions $f\in\mathcal{C}_n(\R)$. Since, by Lemma \ref{setinclussion}, $B_{\infty 1}^n(\R)\subseteq \mathcal{C}_n(\R)$, our result extends \cite[Theorem 7.1]{AlPe11}.
\end{rmrk}

\noindent\textit{Acknowledgment}: A. Chattopadhyay is supported by the Core Research Grant (CRG), File No: CRG/2023/004826, of SERB. S. Giri acknowledges the support by the Prime Minister's Research Fellowship (PMRF), Government of India. C. Pradhan acknowledges support from the Fulbright-Nehru postdoctoral fellowship.

\end{document}